\documentclass[11pt]{amsart}
\usepackage{palatino}
\usepackage{amsmath}
\usepackage{graphicx,epsfig}

\newtheorem{theorem}{$\quad$Theorem}

\newtheorem{definition}[theorem]{$\quad$Definition}

\newcommand{\Z}{{\mathbb Z}}

\theoremstyle{definition}

\begin{document}
\title{Numerical chaos in a fractional order logistic map}

\author{Joakim Munkhammar}

\date{\today}


\maketitle

\begin{center}
\small

 \email{joakim.munkhammar@gmail.com}
 \footnote{Studentstaden
 23:230, 752 33 Uppsala, Sweden}
 \normalsize
\end{center}

\begin{abstract}
\noindent In this paper we investigate a fractional order logistic
map and its discrete time dynamics. We show some basic properties
of the fractional logistic map and numerically study its
period-doubling route to chaos.
\end{abstract}

\vspace{10pt}

\small
\begin{center}
\keywords{2000 AMS Mathematics Subject Classification: 26A33, 37D45, 65P20.\\
Keywords: Fractional Calculus, Dynamical Systems, Numerical Chaos
}
\end{center}
\normalsize

\section{Introduction}
The concept of fractional order differentiation and integration is
nearly as old as calculus itself \cite{Munkhammar1}. Already in
the 17:th century Leibnitz made remarks on the fractional
derivative of order $1/2$ \cite{WHITE}. Despite many fundamental
results and important definitions found more than $150$ years ago
by among others Euler and Riemann, the field of applications of
fractional calculus has mainly recently started to draw some
interest \cite{Cascaval,Laskin,Li,Chen,MARK,Oparnica}.
Applications to physics involve viscoelastic systems and
electro-magnetic waves \cite{Li}, but more recently the dynamics
of fractional order dynamical systems, involving fractional
mechanics and fractional oscillators has been studied
\cite{OSC,Oparnica}. Fractional differential equations have for
example arisen in Quantum Mechanics as well as in fractional
generalizations of Einstein field equations
\cite{Laskin,Munkhammar3}. Also fractional generalizations of
various famous dynamical systems such as the fractional
Lorenz-system, the R\"ossler systems \cite{Li} and Chen's system
\cite{Chen} have been investigated. In this paper we shall
investigate the dynamics of a fractional generalization of the
classic logistic map \cite{Thunberg}.

\section{Dynamics}
The discrete dynamics of unimodal maps is well known and has been
extensively studied \cite{Thunberg}. In particular population
dynamics often consider maps like the Ricker-family
\cite{Thunberg}:
\begin{equation}
R_{\lambda,\beta}(x) = \lambda x e^{-\beta x}, \, \, \lambda>1 \,
\, \beta>0,
\end{equation}
or the Hassel-family (See \cite{Thunberg}):
\begin{equation}
H_{\lambda,\beta}(x) = \frac{\lambda x}{(1+x)^{\beta}}, \, \,
\lambda>1 \, \, \beta>0.
\end{equation}
However the perhaps most famous unimodal map is the logistic map:
\begin{equation}\label{logeq}
Q_\lambda (x)=\lambda x(1-x), \, \, \lambda>0.
\end{equation}
The logistic map is defined on the half-line $x\in [0,\infty[$,
but all of its interesting dynamics takes place on the bounded
interval $I=[0,1]$ for $0< \lambda \le 4$. If we consider discrete
time evolution by the mapping $Q_\lambda : I \rightarrow I$,
letting the state be $x_n$ at time $n$ then we have $x_{n+1} =
Q_\lambda (x_n)$ at time $n+1$. In the prescribed interval the
parameter $\lambda$ is increased which results in a
period-doubling route to chaos. For a thorough background on
unimodal maps and chaos see \cite{Thunberg}.

\section{Fractional Calculus}
There are many models of fractional diff-integration, however the
more popular ones are the Riemann-Liouville-type operator and
Caputo-type operator \cite{Munkhammar2}, whereas the latter one is
favored in cases of lack of initial conditions. We shall use the
Riemann-Liouville fractional operator in order to construct the
fractional logistic map. The Riemann-Liouville fractional integral
is defined as follows.

\begin{definition}\label{RLINTEGRAL}
 If $f(x)\in C([a,b])$ and $a<x<b$ then
 \begin{equation}\label{RLIN}
   I_{a+}^{\alpha} f(x) := \frac {1}{\Gamma (\alpha)} \int_{a}^{x}
   \frac{f(t)}{(x-t)^{1-\alpha}} \, dt
 \end{equation}

 where $\alpha\in ]-\infty,\infty[ $, is called the Riemann-Liouville
 fractional integral of order $\alpha$. In the same
 fashion for $\alpha\in ]0,\infty[$ we let
 \begin{equation}\label{RLDERIVATIVE}
   D_{a+}^{\alpha} f(x) := \frac{1}{\Gamma (1-\alpha)} \frac{d}{dx}
   \int_{a}^{x} \frac{f(t)} {(x-t)^{\alpha}} \, dt,
 \end{equation}
 which satisfies
 \begin{equation}
    D_{a+}^0 f(x) = I_{a+}^0 f(x) =f(x),
 \end{equation}
and is called the Riemann-Liouville fractional derivative of order
$\alpha$.
\end{definition}
The fractional derivative is a special-case of the fractional
integral. A fractional derivative of order $1/2$ is called a
semi-derivative and a fractional integral of the same order is
called semi-integral \cite{Munkhammar2}. Both the fractional
derivative and fractional integral satisfies the important
semi-group property:
\begin{theorem}\label{SEMIGROUP}
 For any $f \in C([a,b])$ the Riemann-Liouville fractional integral satisfies
 \begin{equation}
   I_{a+}^{\alpha} I_{a+}^\beta f(x)= I_{a+}^{\alpha+\beta}f(x)
 \end{equation}
 for $\alpha>0,\beta>0$.
\end{theorem}
A proof may be found in \cite{Munkhammar2}. For more information
regarding fractional calculus and its applications see
\cite{Munkhammar1,Munkhammar2,WHITE}.

\section{Fractional dynamics}

\subsection{Basic properties of the fractional logistic map}
Many investigations of fractional dynamical systems contain
fractional derivatives in time which is due to their continuous
time derivative (See \cite{OSC,Li,Chen}). There exists, however,
fractional differential equations in spatial sense, take the
fractional Schr\"odinger equation for example \cite{Laskin}. In
our fractional generalization of the logistic map we shall use a
spatial fractional derivative. If we let the Riemann-Liouville
fractional integral \eqref{RLDERIVATIVE} operate on the logistic
map \eqref{logeq} we get:
\begin{equation}
I_{0+}^{\alpha} Q_\lambda (x) = \frac{\lambda}{\Gamma(\alpha +2)}
 \Big(1-\frac{2x}{\alpha+2}\Big)x^{1+\alpha},
\end{equation}
which is valid for $\alpha \in\, ]0, \infty[$ (for more details on
fractional differentiation see \cite{Munkhammar1}), thus
\begin{definition}\label{fraclogdef}
For all $x \geq 0$, $\lambda>0$ and $\alpha \in\, ]0, \infty[$:
\begin{equation}
Q_\lambda^\alpha (x) := \frac{\lambda}{\Gamma(\alpha +2)}
 \Big(1-\frac{2x}{\alpha+2}\Big)x^{1+\alpha},
\end{equation}
will be called the fractional logistic map (FLM) of order
$\alpha$.
\end{definition}
It is left as an exercise to show that the ordinary logistic map
\eqref{logeq} follows as a special-case when $\alpha=0$. The
Riemann-Liouville fractional operators have special properties for
the parameter-value $\alpha=1/2$ and are in these cases called
semi-integral and semi-derivative respectively
\cite{Munkhammar1,Munkhammar2,WHITE}. We shall therefore give the
definition:
\begin{definition}\label{semilogistic}
For $x \geq 0$ and $\lambda>0$ the $\alpha=1/2$ order fractional
logistic map appears like:
\begin{equation}
Q_\lambda^{1/2} (x)= \frac{4\lambda}{3\sqrt{\pi}} \Big(1-
\frac{4x}{5} \Big) x^{3/2},
\end{equation}
and will be called the semi-logistic map.
\end{definition}
Due to the special property of fractional diff-integrals we can
obtain the following useful theorem:
\begin{theorem}\label{Semitwo}
For all $\alpha, n \in ]0,\infty[$:
\begin{equation}
\frac{d^n}{dx^n} Q_\lambda^\alpha (x) = Q_\lambda^{\alpha-n} (x).
\end{equation}
\end{theorem}
\begin{proof}
It follows directly from definition \ref{fraclogdef} and the
semi-group property (Theorem \ref{SEMIGROUP}).
\end{proof}
%
%
Despite the fact that the fractional logistic map is valid for
$\alpha \in [0, \infty]$ it is worth mentioning that probably the
most interesting dynamics takes place for $\alpha \in [0,1]$ since
theorem \ref{Semitwo} allows for translation to any other domain
by $N$-times integration/differentiation for any $N \in \Z$. We
can also give the following theorem:
\begin{theorem}\label{fundfrac}
For $\alpha \in [0,1]$:
\begin{equation}
Q_\lambda^\alpha (0) = Q_\lambda^\alpha(1+\frac{1}{2}\alpha) = 0,
\end{equation}
holds for all $\lambda \geq 0$.
\end{theorem}
\begin{proof}
$Q_\lambda^\alpha (0) =0$ follows directly from definition
\ref{fraclogdef}. Also we find that
$Q_\lambda^\alpha(1+\frac{1}{2}\alpha) = 0$ by inserting $x =
1+\frac{1}{2}\alpha$:
\begin{equation}
Q_\lambda^\alpha (1+\frac{1}{2}\alpha) =
\frac{\lambda}{\Gamma(\alpha +2)} \Big(1-\frac{2+\alpha}{\alpha
+2} \Big) \Big(1+\frac{1}{2}\alpha \Big)^{1+\alpha} = 0.
\end{equation}
Note that $\lambda \geq 0$ is a part of the definition of the
fractional logistic map (Definition \ref{fraclogdef}).
\end{proof}
In order to find fixed points of the fractional logistic map we
need the following criterion to be fulfilled:
\begin{equation}
Q_\lambda^\alpha(x) = x,
\end{equation}
which is an equation on the form:
\begin{equation}\label{fundeq}
x^{1+\alpha} - \frac{\alpha+2}{2}x^\alpha +
\frac{\Gamma(\alpha+3)}{2\lambda} = 0.
\end{equation}
This equation has non-fractional solutions for $\alpha \in
[1,2,3,...]$ which corresponds to simple integer integrations of
the logistic map. For the fractional values a general solution is
left as an open problem. Below we numerically study the dynamics
of the fractional logistic map.

\subsection{Period doubling and chaos}\label{calc}
With the aid of \textit{Mathematica 5} and software developed in
\cite{Knapp} we were able to numerically study the fractional
logistic map. As an example we plot the fractional logistic map in
the domain $x \in [0,1]$, $\alpha \in [0,1]$:
\begin{center}\label{fig1}
\epsfig{file=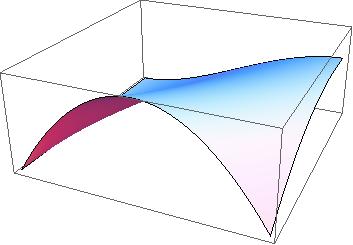,height=146pt,width=212pt}\\
\small Fig.1. $\alpha \in [0,1]$, $x\in [0,1]$. \normalsize
\end{center}
Note that $x=1$, $\alpha=1$ is a local maximum in the
$x$-direction since the fractional logistic map for $x=1$ and
$\alpha=0$ is zero due to the property of fractional
diff-integrals in theorem \ref{Semitwo}. As an example we perform
three iterates and end up with the following plot:
%
%
\begin{center}\label{fig2}
\epsfig{file=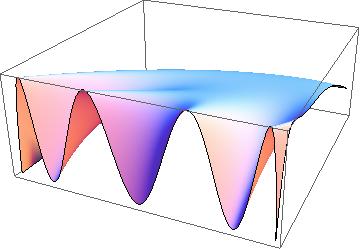,height=146pt,width=212pt}\\
\small Fig.2. $\alpha \in [0,1]$, $x\in [0,1]$. \normalsize
\end{center}
In figure $2$ we see
$Q_\lambda^\alpha(Q_\lambda^\alpha(Q_\lambda^\alpha(x)))$. It is
also worth mentioning that the general appearance for any choice
of $\alpha \in [0,1[$ appears similar to the ordinary logistic map
(which is the special case when $\alpha =0$). Indeed, if we
further examine the special case $\alpha=1/2$, the semi-logistic
map, then we see that numerically a period-doubling route to chaos
appears as $\lambda$ is increased from $4.5$ to $6.1$:
\begin{center}\label{fig3}
\epsfig{file=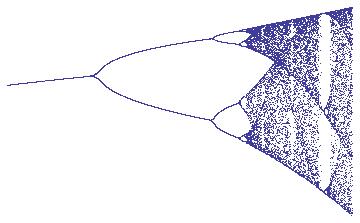,height=146pt,width=212pt}\\
\small Fig.3. $\alpha =1/2$, $x\in [0,3]$. \normalsize
\end{center}
The appearance of this bifurcation diagram is similar to that
observed for the logistic map \cite{Thunberg}. Thus we have hinted
that the fractional logistic map exhibits a period-doubling route
to chaos.

\section{Discussion and open problems}
In this paper we studied the discrete time evolution of a
fractional generalization of the logistic map; the fractional
logistic map. We showed some basic properties of the fractional
logistic map and numerically hinted a period-doubling route to
chaos for the special case $\alpha = 1/2$. A rigorous proof of
this is left for further investigation. The fractional parameter
$\alpha$ in the fractional logistic map brings on a whole new
range of possibilities for traditional discrete dynamics. The
possibility of showing that "fractional chaos" arises when letting
the traditional bifurcation parameter be constant while varying
the fractional parameter is left as an interesting open problem.
Furthermore, the dynamics of the fractional logistic map in the
complex plane is also left as an open issue for further
investigation. The connection between fractional discrete time
dynamics, like the fractional logistic map, and other fractional
dynamical systems such as the fractional R \"ossler and fractional
Lorenz systems are also open issues.




\begin{thebibliography}{99}

\bibitem{OSC}
B. N. N. Achar, J. W. Hanneken, T. Enck, T.Clarke,
\textit{Dynamics of the fractional oscilator}, Physica A 297
(2001), 361-367.

\bibitem{Cascaval}
R. C. Cascaval, E. C. Eckstein, C. L. Frota, J. A. Goldstein,
\textit{Fractional telegraph equations}, J. Math. Anal. Appl. 276
(2002), 145-159.

\bibitem{Devaney}
R. L. Devaney, \textit{Chaotic Dynamical Systems}, Westview press
(2003).

\bibitem{Knapp}
R. Knapp, M. Sofroniou, \textit{Some Numerical Aspects of
Functional Iteration and Chaos in Mathematica},
http://library.wolfram.com/infocenter, Mathematica Information
center (2005).

\bibitem{Laskin}
N. Laskin, \textit{Fractional Schr\"odinger equation}, Phys. Rev.
E 66, 056108 (2002).

\bibitem{Li}
C. Li, G. Chen, \textit{Chaos and Hyper chaos in the fractional
order Rö\"ossler equations}, Physica A 341 (2004), p.55-61.

\bibitem{Chen}
C. Li, G. Peng, \textit{Chaos in Chen's system with a fractional
order}, Chaos Sol. frac. 22 (2004) p.443-450.

\bibitem{MARK}
M. M. Meerschaert, \textit{The fractional calculus project}, MAA
Student Lecture, Phoenix January (2004).

\bibitem{Munkhammar1}
J. D. Munkhammar, \textit{Riemann-Liouville fractional derivatives
and the Taylor-Riemann series}, UUDM project report 2004:7 (2004).

\bibitem{Munkhammar2}
J. D. Munkhammar, \textit{Fractional Calculus and the
Taylor-Riemann series}, RHIT U. J. Math. 6(1) (2005).

\bibitem{Munkhammar3}
J. D. Munkhammar, \textit{Riemann-Liouville Fractional Einstein
Field Equations},  arXiv:1003.4981v1 (2010).

\bibitem{Oparnica}
L. Oparnica, \textit{Generalized fractional calculus with
applications in mechanics}, Matematicki Vesnik 53 (2001)
p.151-158.

\bibitem{rossler}
O. E. R\"ossler, \textit{An equation for continuous chaos}, Phys.
lett. A57 (1976), p.397-399.

\bibitem{WHITE}
S. G. Samko, A. A. Kilbas, O. I. Marichev, \textit{Fractional
integrals and derivatives: theory and applications}, Gordon and
Breach, Amsterdam, (1993).

\bibitem{Strogatz}
S. H. Strogatz, \textit{Non-linear dynamics and chaos}, Perseus
Books publ. (1994).

\bibitem{Thunberg}
H. Thunberg, \textit{Periodicity versus chaos in one-dimensional
dynamics}, SIAM Review Vol. 43 No. 1 (2001), p.3-30.

\end{thebibliography}
\end{document}